\documentclass[letterpaper, 10pt, twocolumn]{ieeeconf}      

\IEEEoverridecommandlockouts                              
\usepackage{amsmath,mathrsfs,amsfonts,amssymb,graphicx,epsfig,stmaryrd}
 \usepackage{amsthm}
\usepackage{subcaption}
\usepackage{caption}
\usepackage{color,multirow,rotating}
\usepackage{algorithm,algpseudocode,algorithmicx}
\usepackage{cite,url,framed,bm,balance,dsfont,varwidth}
\usepackage{nicefrac}

\newtheorem{theorem}{Theorem}

\newtheorem{remark}{Remark}

\usepackage{cleveref}
\usepackage{tikz}
\usetikzlibrary{calc,trees,positioning,arrows,chains,shapes.geometric,decorations.pathreplacing,decorations.pathmorphing,shapes, matrix,shapes.symbols}

\DeclareMathOperator*{\erf}{erf}
\DeclareMathOperator*{\esssup}{ess\,sup}

\newcommand{\differential}{{\rm{d}}}
\newcommand{\Phicc}{\ensuremath{\Phi_{\rm{cc}}}}
\newcommand{\Phice}{\ensuremath{\Phi_{\rm{ce}}}}
\newcommand{\phicc}{\ensuremath{\phi_{\rm{cc}}}}
\newcommand{\phice}{\ensuremath{\phi_{\rm{ce}}}}
\newcommand{\Ccc}{\ensuremath{C_{\rm{cc}}}}
\newcommand{\Cce}{\ensuremath{C_{\rm{ce}}}}

\allowdisplaybreaks

\title{\LARGE\textbf{
A Controlled Mean Field Model for Chiplet Population Dynamics}
}

\author{Iman Nodozi, Abhishek Halder, Ion Matei
\thanks{Iman Nodozi is with the Department of Electrical and Computer Engineering, University of California, Santa Cruz, CA 95064, USA,
        {\tt\small{inodozi@ucsc.edu}}.\\
        Abhishek Halder is with the Department of Applied Mathematics, University of California, Santa Cruz, CA 95064, USA,
        {\tt\small{ahalder@ucsc.edu}}.\\
        Ion Matei is with the Palo Alto Research Center, Inc., Palo Alto, CA 94304, USA,
        {\tt\small{imatei@parc.com}}.\\
        This work was partially supported by NSF grant 2112755.
}}

\begin{document}

\maketitle
\pagenumbering{arabic}

\bstctlcite{IEEEexample:BSTcontrol} 

\begin{abstract}
In micro-assembly applications, ensemble of chiplets immersed in a dielectric fluid are steered using dielectrophoretic forces induced by an array of electrode population. Generalizing the finite population deterministic models proposed in prior works for individual chiplet position dynamics, we derive a controlled mean field model for a continuum of chiplet population in the form of a nonlocal, nonlinear partial differential equation. The proposed model accounts for the stochastic forces as well as two different types of nonlocal interactions, viz. chiplet-to-chiplet and chiplet-to-electrode interactions. Both of these interactions are nonlinear functions of the electrode voltage input. We prove that the deduced mean field evolution can be expressed as the Wasserstein gradient flow of a Lyapunov-like energy functional. With respect to this functional, the resulting dynamics is a gradient descent on the manifold of joint population density functions with finite second moments that are supported on the position coordinates.
\end{abstract}


\section{Introduction}\label{sec:introduction}
This work is motivated by micro-assembly applications, such as printer systems \cite{matei2017towards,matei2019micro} and manufacturing of photovoltaic solar cells, where an array of electrodes can be used to generate spatio-temporally non-homogeneous electric potential landscapes for dynamically assembling the ``chiplets"--micron sized particles immersed in dielectric fluid--into desired patterns. In such applications, the electric potentials generated by the array of electrodes induce non-uniform dielectrophoretic forces on the chiplets, thereby resulting in a population-level chiplet dynamics. The purpose of the present work is to propose a controlled mean field model for the same.

There have been several works \cite{lu2014open,edwards2014controlling,matei2020micro,matei2021micro,lefevre2022closed} on the modeling and dielectrophoretic control of chiplet population. However, a continuum limit macroscopic dynamics that accounts for both chiplet-to-chiplet and chiplet-to-electrode nonlocal interactions, as considered herein, has not appeared before.

The mean field limit pursued here involves considering the number of chiplets and electrodes as infinity, i.e., to think both of them as continuum population. There are two reasons why this could be of interest. \emph{First,} the continuum limit helps approximate and better understand the dynamics for large but finitely many chiplets and electrodes, which is indeed the situation in the engineering applications mentioned before. \emph{Second,} the distributed control synthesis problem for large but finite population becomes computationally intractable, as noted in recent works \cite{matei2021micro,matei20222d,matei2022system}. A controlled mean field model opens up the possibility of designing a controller in the continuum limit with optimality guarantees. Such a controller can then be applied to a large but finite population with sub-optimality bounds. We clarify here that in this work, we only present the mean field model and its properties. We leave the control synthesis problem for our follow up work.   

As in prior works such as \cite{matei2021micro}, we consider the chiplet dynamics in two dimensional position coordinate. Specifically, let $\bm{x}(t)\in\mathbb{R}^2$ denote the position vector of a chiplet at any fixed time $t\in[0,\infty)$, and let 
$$u: \mathbb{R}^2 \times [0,\infty) \mapsto [u_{\min}, u_{\max}]\subset\mathbb{R}$$ 
denote a causal deterministic control policy, i.e., $u=u(\bm{x},t)$. The control $u$ represents the electrode voltage input, and in practice, the typical voltage range $[u_{\min}, u_{\max}]=[-400, 400]$ Volt. We denote the collection of admissible control policies as $\mathcal{U}$. For a typical experimental set up detailing the sensing-control architecture, see \cite[Sec. II]{matei2021micro}.

A viscous drag force balances the controlled force vector field $\bm{f}^{u}$ induced by the joint effect of the chiplet-to-chiplet and chiplet-to-electrode interactions. At the low Reynolds number context relevant here, the viscous drag force is proportional to $\dot{\bm{x}}$, where the proportionality constant $\mu$ denotes the viscous coefficient of the dielectric fluid. Ignoring the acceleration due to negligible mass of a chiplet, the dynamics then takes a form
\begin{align}
\underbrace{\mu\dot{\bm{x}}}_{\text{viscous drag force}} = \underbrace{\bm{f}^{u}}_{\text{controlled interaction force}} + \quad \text{noise}
\label{ChipletDynamicsHighLevel}
\end{align}
where the noise may result from stochastic forcing due to environmental fluctuations (e.g., dielectric fluid impurities) and/or unmodeled dynamics.

\subsubsection*{Contributions} In this paper, we make the following two specific contributions.
\begin{itemize}
\item We derive a controlled mean field dynamics (Sec. \ref{sec:Model}) for the macroscopic motion of the chiplet population. The derived model is non-affine in control, and rather non-standard compared to the existing nonlocal dynamics models available in the literature. 

\item We establish that the derived mean field dynamics model can be understood as the Wasserstein gradient flow (Sec. \ref{sec:WassGradFlow}) of a free energy functional over the manifold of chiplet population density functions. 
\end{itemize}


\section{Notations and Preliminaries}
\label{sec:prelim}
\noindent{\textbf{Wasserstein distance.}} The Wasserstein distance $W$ between a pair of probability density functions $\rho_1(\bm{x}),\rho_2(\bm{y})$ (or between corresponding probability measures in general) with finite second moments, respectively supported on $\mathcal{X},\mathcal{Y}\subseteq\mathbb{R}^{d}$, is defined as
\begin{align}
W\!\left(\rho_1,\rho_2\right)\! :=\!\! \left(\underset{\rho\in\Pi_2\left(\rho_1,\rho_2\right)}{\inf}\!\!\int_{\mathcal{X}\times\mathcal{Y}}\!\!\!\|\bm{x}-\bm{y}\|_{2}^{2}\:\rho(\bm{x},\bm{y})\differential\bm{x}\differential\bm{y}\right)^{\frac{1}{2}}  
\label{defWass}    
\end{align}
where $\Pi_2\left(\rho_1,\rho_2\right)$ is the collection of all joint probability density functions $\rho(\bm{x},\bm{y})$ supported on the product space $\mathcal{X}\times\mathcal{Y}$ having finite second moments, $\bm{x}$ marginal $\rho_1$, and $\bm{y}$ marginal $\rho_2$. As such, \eqref{defWass} involves an infinite dimensional linear program that goes back to the work of Kantorovich \cite{kantorovich1942translocation}. It is well-known \cite[p. 208]{villani2003topics} that $W$ is a metric on the space of probability density functions (more generally, on the space of probability measures). Under mild assumptions, the minimizing measure $\rho^{\rm{opt}}(\bm{x},\bm{y})\differential\bm{x}\differential\bm{y}$ is supported on the graph of the optimal transport map $T^{\rm{opt}}:\mathcal{X}\mapsto\mathcal{Y}$ pushing the measure $\rho_1(\bm{x})\differential\bm{x}$ forward to $\rho_2(\bm{y})\differential\bm{y}$. For many connections between the Wasserstein metric and theory of optimal mass transport, we refer the readers to \cite{villani2003topics,villani2009optimal}.

\noindent{\textbf{Wasserstein gradient of a functional.}}
Let $\mathcal{P}\left(\mathbb{R}^d\right)$ denote the space of all probability density functions supported over the subsets of $\mathbb{R}^{d}$, and denote the collection of probability density functions with finite second moments as $\mathcal{P}_2\left(\mathbb{R}^d\right)\subset \mathcal{P}\left(\mathbb{R}^d\right)$. The \emph{Wasserstein gradient} of a functional $\Phi:\mathcal{P}_2\left(\mathbb{R}^d\right) \mapsto \mathbb{R}$, denoted as $\nabla^{W}\Phi$, evaluated at $\rho\in\mathcal{P}_2\left(\mathbb{R}^d\right)$, is given by \cite[Ch. 8]{ambrosio2008gradient}
\begin{align}
    \nabla^{W}\Phi\left(\rho\right) := -\nabla \cdot\left(\rho \nabla \frac{\delta \Phi}{\delta \rho}\right)
    \label{equ:Wasserstein gradient}
\end{align}
where $\nabla$ denotes the standard Euclidean gradient, and $\frac{\delta}{\delta \rho}$ denotes the functional derivative w.r.t. $\rho$.

To exemplify the definition \eqref{equ:Wasserstein gradient}, consider the functional $\Phi(\rho)=\int\rho\log\rho$ (negative entropy) for $\rho\in\mathcal{P}_2\left(\mathbb{R}^d\right)$. Then $\frac{\delta\Phi}{\delta \rho}=1 + \log\rho$, $\nabla(1+\log\rho)=\nabla\rho/\rho$, and we get $\nabla^{W}\Phi\left(\rho\right)=-\nabla\cdot\nabla\rho = -\Delta \rho$, where $\Delta := \nabla\cdot\nabla$ denotes the Euclidean Laplacian operator.

\noindent{\textbf{Other notations.}} The notation $\langle\cdot,\cdot\rangle$ is used to denote either the standard Euclidean inner product of vectors, or the $L^2$ inner product of functions, as evident from the context. For any natural number $n$, we use the finite set notation $\llbracket n\rrbracket:=\{1,2,\hdots,n\}$. The symbols $\esssup$, $\mathbb{E}$, $\mathbb{P}$, $\bm{I}_2$, $\|\cdot\|_2$ and $\|\cdot\|_{\infty}$ denote the essential supremum, the expectation, the probability measure, the $2\times 2$ identity matrix, the vector 2 and $\infty$ norms, respectively. The symbol $\sim$ is used as a shorthand for ``follows the statistical distribution density".

Given probability measures $\mu_0,\mu_1$ on $\mathbb{R}^{d}$, the \emph{total variation distance}
${\rm{dist}}_{\rm{TV}}(\mu_0,\mu_1):=\frac12\sup_{f}\left|\int f\:{\rm d}(\mu_0 - \mu_1)\right|$
where the supremum is over all measurable $f:\mathbb{R}^{d}\rightarrow\mathbb{R}$, $\|f\|_{\infty}\leq 1$. For $f:\mathbb{R}^{d}\rightarrow\mathbb{R}$, we define its \emph{Lipschitz constant} $\|f\|_{\rm{Lip}}:=\sup_{\bm{x}\neq\bm{y}}\frac{|f(\bm{x})-f(\bm{y})|}{\|\bm{x}-\bm{y}\|_2}$, and its \emph{bounded Lipschitz constant} $\|f\|_{\rm{BL}}:=\max\{\|f\|_{\infty},\|f\|_{\rm{Lip}}\}$. The \emph{bounded Lipschitz distance} \cite[Ch. 11.3]{dudley2002real} between probability measures $\mu_0,\mu_1$ is ${\rm{dist}}_{\rm{BL}}(\mu_0,\mu_1):=\sup_{\|f\|_{\rm{BL}}\leq 1}\left|\int f\:{\rm d}(\mu_0 - \mu_1)\right|$. Notice that ${\rm{dist}}_{\rm{BL}}(\mu_0,\mu_1)\leq 2\:{\rm{dist}}_{\rm{TV}}(\mu_0,\mu_1)$.

For $\mathcal{X}\subseteq \mathbb{R}^{d}$, we use $C_b(\mathcal{X})$ to denote the space of all bounded continuous functions $\varphi:\mathcal{X}\mapsto\mathbb{R}$, and $C_b^k(\mathcal{X})$ comprises those which are also $k$ times continuously differentiable (in the sense of mixed partial derivatives of order $k$). We say that a function sequence $\{g_n\}_{n\in\mathbb{N}}$ where $g_n\in L^{1}(\mathcal{X})$, converges weakly to a function $g\in L^{1}(\mathcal{X})$, if 
$\lim_{n\rightarrow\infty}\int_{\mathcal{X}}\left(g_n - g\right)\psi = 0$ for all $\psi\in C_b(\mathcal{X})$. We symbolically denote the weak convergence as $g_n \rightharpoonup g$.


\section{Controlled Mean Field Model}\label{sec:Model}

In this Section, we introduce the chiplet population dynamics. Such model has its origin in the physical processes enabling silicon microchips to be manipulated by both electrophoretic and dielectrophoretic forces when they are placed in dielectric carriers such as Isopar-M \cite{4891957}. These carriers have low conductivity which allows long-range Coulomb interactions. In general, the dielectrophoretic forces dominate, and they are induced by the potential energy generated by electrostatic potentials created in electrodes. The electrodes are formed by depositing nm-scale Molybdenum-Chromium (MoCr) onto a glass substrate via vapor deposition and then directly patterning them with a laser ablation tool. The electrodes are then insulated from the chiplets and dielectric fluid by thermally laminating a micrometer-scale thick perfluoroalkoxy (PFA) film. The dielectric forces act on the chiplets, while viscous drag forces proportional to their velocities oppose their motion. Due to the negligible mass of the chiplets, their acceleration can be ignored.

Let us denote the \emph{normalized chiplet population density function} (PDF) at time $t$ as $\rho(\bm{x},t)$. By definition, $\rho\geq 0$ and $\int_{\mathbb{R}^2}\rho\:\differential\bm{x} = 1$ for all $t$. 

We make the following assumptions.
\begin{itemize}
    \item[\textbf{A1}.] Under an admissible control policy $u\in\mathcal{U}$, the chiplet normalized population distribution over the two dimensional Euclidean configuration space remains absolutely continuous w.r.t. the Lebesgue measure $\differential\bm{x}$ for all $t\in[0,\infty)$. In other words, the corresponding PDFs $\rho(\bm{x},t)$ exist for all $t\in[0,\infty)$.
    
    \item[\textbf{A2}.] Under an admissible control policy $u\in\mathcal{U}$, we have $\rho \in \mathcal{P}_2(\mathbb{R}^2)$ for all $t$.
\end{itemize}

The sample path dynamics of a chiplet position is governed by a controlled nonlocal vector field $$\bm{f}^u:\mathbb{R}^2 \times [0,\infty) \times \mathcal{U} \times \mathcal{P}_2(\mathbb{R}^2)\mapsto\mathbb{R}^{2}$$ 
induced by a controlled \emph{interaction potential} $\phi^u:\mathbb{R}^2 \times \mathbb{R}^2\times[0,\infty) \mapsto \mathbb{R}$, i.e., 
\begin{align}
\bm{f}^u(\bm{x},t,u,\rho) := -\nabla\left(\rho * \phi^u\right),
\label{Controlledvectorfield}    
\end{align}
where $*$ denotes \emph{generalized  convolution} in the sense 
$$\left(\rho * \phi^u\right)(\bm{x},t) := \int_{\mathbb{R}^2}\phi^u(\bm{x},\bm{y},t)\rho(\bm{y},t)\differential\bm{y}.$$ 
The superscript $u$ in $\phi^{u}$ emphasizes that the potential depends on the choice of control policy.
In particular,
\begin{subequations}
\begin{align}
\!\!\!\!\phi^u(\bm{x},\bm{y},t) &:= \phicc^u(\bm{x},\bm{y},t) + \phice^u(\bm{x},\bm{y},t), \label{U}\\
\!\!\!\!\phicc^u(\bm{x},\bm{y},t) &:= \frac{1}{2}\Ccc\left(\|\bm{x}-\bm{y}\|_2\right)\left(\bar{u}(\bm{y},t) - \bar{u}(\bm{x},t)\right)^2,\label{Ucc}\\
\!\!\!\!\phice^u(\bm{x},\bm{y},t) &:= \frac{1}{2}\Cce\left(\|\bm{x}-\bm{y}\|_2\right) \left(u(\bm{y},t) - \bar{u}(\bm{x},t)\right)^2, \label{Uce}
\end{align}
\label{InteractionPotential}
\end{subequations}
for $\bm{x},\bm{y}\in\mathbb{R}^2$ and 
\begin{align}
 \bar{u}(\bm{x},t) :=\frac{\int_{\mathds{R}^2}C_{\rm{ce}}\left(\|\bm{x}-\bm{y}\|_2\right)u(\bm{y},t)\rho (\bm
{y}, t) \differential \bm{y}}{\int_{\mathds{R}^2}C_{\rm{ce}}\left(\|\bm{x}-\bm{y}\|_2\right) \rho (\bm
{y}, t)\differential \bm{y}}. 
\label{defubar}
\end{align}
The subscripts {\rm{cc}} and {\rm{ce}}  denote the chiplet-to-chiplet and chiplet-to-electrode interactions, respectively. As before, the superscript $u$ highlights the dependence on the choice of control policy. In \eqref{Ucc}-\eqref{Uce}, $\Ccc$ and $\Cce$ respectively denote the chiplet-to-chiplet and chiplet-to-electrode capacitances. These capacitances can be determined using two dimensional electrostatic COMSOL\textsuperscript{\textregistered} \cite{comsol} simulations for a symmetric chiplet geometry. Such simulation model comprises two metal plates with dimensions defined by the chiplet and electrode geometry, surrounded by a dielectric with properties identical to those of the Isopar-M solution.
The capacitances are computed from the charges that result on each conductor when an electric potential is applied to one and the other is grounded. Once the capacitance among chiplets and electrodes at different distances are computed, differentiable parameterized capacitance function approximations (e.g., linear combination of error functions) can be fitted to that data. 

In words, \eqref{U} says that the total controlled interaction potential $\phi^{u}$ is a sum of the chiplet-to-chiplet interaction potential $\phicc^u$ given by \eqref{Ucc}, and the chiplet-to-electrode interaction potential $\phice^u$ given by \eqref{Uce}.

The expressions for \eqref{Ucc}, \eqref{Uce}, \eqref{defubar} arise from capacitive electrical circuit abstraction that lumps the interaction between the electrodes and the chiplets. In \cite[Sec. III]{matei2021micro}, such an abstraction was detailed for a finite population of $n$ chiplets and $m$ electrodes. The expressions \eqref{Ucc}, \eqref{Uce}, \eqref{defubar} generalize those in the limit $n,m\rightarrow\infty$. On the other hand, specializing \eqref{Ucc}, \eqref{Uce}, \eqref{defubar} for a finite population $\{\bm{x}_{i}\}_{i\in\llbracket n\rrbracket}$ with $\rho\equiv\frac{1}{n}\sum_{i=1}^{n}\delta_{\bm{x}_i}$ where $\delta_{\bm{x}_{i}}$ denotes the Dirac delta at $\bm{x}_{i}\in\mathbb{R}^2$, indeed recovers the development in \cite[Sec. III]{matei2021micro}.

\begin{remark}\label{RemarkPotentialSymmetry}
An immediate observation from \eqref{InteractionPotential} is that even though the potential $\phicc^u$ is symmetric in $\bm{x},\bm{y}$, the potential $\phice^u$ is not. Therefore, the overall controlled interaction potential $\phi^{u}$ is not symmetric in $\bm{x},\bm{y}$. 
\end{remark}

Without loss of generality, we assume unity viscous coefficient in \eqref{ChipletDynamicsHighLevel}, i.e., $\mu=1$ (since otherwise we can re-scale the $\bm{f}^u$). In addition, assuming the chiplet velocity is affected by additive standard Gaussian White noise, the sample path dynamics of the $i$th chiplet position $\bm{x}_{i}(t)$ then evolves as per 
a controlled interacting diffusion, i.e., as a It\^{o} stochastic differential equation (SDE) with \emph{nonlocal} nonlinear drift:
\begin{align}
\differential\bm{x}_{i} = \bm{f}^u(\bm{x}_{i},t,u,\rho)\:\differential t + \sqrt{2\beta^{-1}}\:\differential\bm{w}_{i}(t), \quad i\in\llbracket n\rrbracket,
\label{ItoSDE}    
\end{align}
where $\bm{f}^u$ is given by \eqref{Controlledvectorfield}, $\beta>0$ denotes inverse temperature, and $\bm{w}_{i}(t)\in\mathbb{R}^{2}$ denote i.i.d. realizations of a standard Wiener process that is $\mathcal{F}_t$-adapted on a complete filtered probability space 
with sigma-algebra $\mathcal{F}$ and associated filtration $\left(\mathcal{F}_t\right)_{t\geq 0}$. In particular, $\mathcal{F}_0$ contains all $\mathbb{P}$-null sets and $\mathcal{F}_t$ is right continuous.

The study of SDEs with nonlocal nonlinear drift originated in \cite{mckean1966class}, and has grown into a substantial literature, see e.g., \cite{sznitman1991topics,carmona2018probabilistic}. In statistical physics, such models are often referred to as ``propagation of chaos"--a terminology due to Kac \cite{kac1956foundations}. A novel aspect of the model \eqref{ItoSDE} w.r.t. the existing literature is that the interaction potential $\phi^{u}$ has a nonlinear dependence on the control policy $u(\bm{x},t)$ as evident from \eqref{InteractionPotential}.

\subsection{Existence-Uniqueness of Solution for \eqref{ItoSDE}}
For a given causal control policy $u\in\mathcal{U}$, it is known \cite[Thm. 2.4]{lacker} that an interacting diffusion of the form \eqref{ItoSDE} with initial condition $\bm{x}_{i0}\sim \rho_0$ admits unique weak solution provided the following four conditions hold:\\
(i) the drift $\bm{f}^{u}$ is jointly Borel measurable w.r.t. $\mathbb{R}^{2}\times[0,\infty)\times\mathcal{P}\left(\mathbb{R}^2\right)$,\\
(ii) the diffusion coefficient $\sqrt{2\beta^{-1}}\bm{I}_{2}$ is invertible, and the driftless SDE $\differential\bm{z}(t) = \sqrt{2\beta^{-1}}\differential\bm{w}(t)$ admits unique strong solution,\\
(iii) the drift $\bm{f}^{u}$ is uniformly bounded,\\
(iv) there exists $\kappa>0$ such that
\begin{align*}
&\|\bm{f}^u\left(\bm{x},t,u(\bm{x},t),\rho\right) - \bm{f}^u\left(\bm{x},t,u(\bm{x},t),\widetilde{\rho}\right)\|_2\\
&\quad\leq \kappa \;{\rm{dist}}_{\rm{TV}}\left(\rho,\widetilde{\rho}\right)\quad\text{uniformly in}\; (\bm{x},t)\in\mathbb{R}^{2}\times[0,\infty).
\end{align*}
We assume that the capacitances $\Ccc,\Cce$ in \eqref{InteractionPotential}-\eqref{defubar} are sufficiently smooth, and the control $u$ can be parameterized to ensure smoothness for guaranteeing that $\nabla_{\bm{x}}\phicc^u,\nabla_{\bm{x}}\phice^u$ (and thus $\nabla_{\bm{x}}\phi^u$) are $\|\cdot\|_2$ Lipschitz and uniformly bounded.

As $\nabla_{\bm{x}}\phi^u$ is bounded,  $\bm{f}^{u}=\smallint_{\mathbb{R}^2}\nabla_{\bm{x}}\phi^{u}(\bm{x},\bm{y},t)\rho(\bm{y})\differential\bm{y}$, which being an average of Lipschitz, is itself Lipschitz and thus continuous. Since $\boldsymbol{f}^u$ is continuous, the preimage of any Borel set in $\mathbb{R}^2$ under $\boldsymbol{f}^u$ is a measurable set in $\mathbb{R}^2 \times [0,\infty) \times \mathcal{U} \times \mathcal{P}_2(\mathbb{R}^2)$. Thus, condition (i) holds.

Condition (ii) holds for any $\beta > 0$ since $\bm{z}(t)$ is a Wiener process with variance $2\beta^{-1}$.

For (iii), we find $\underset{(\bm{x},t)\in\mathbb{R}^{2}\times [0,\infty)]}{\esssup}\|\bm{f}^u\left(\bm{x},t,u(\bm{x},t),\rho\right)\|_{\infty}$
\begin{align}
=& \esssup_{(\bm{x},t)\in\mathbb{R}^{2}\times [0,\infty)]}\|\smallint_{\mathbb{R}^2}\nabla_{\bm{x}}\phi^{u}(\bm{x},\bm{y},t)\rho(\bm{y})\differential\bm{y}\|_{\infty}\nonumber\\
\leq& \esssup_{(\bm{x},t)\in\mathbb{R}^{2}\times [0,\infty)]}\smallint_{\mathbb{R}^2}\|\nabla_{\bm{x}}\phi^{u}(\bm{x},\bm{y},t)\rho(\bm{y})\|_{\infty}\differential\bm{y}\nonumber\\
\leq& \smallint_{\mathbb{R}^2}\esssup_{(\bm{x},t)\in\mathbb{R}^{2}\times [0,\infty)]}\|\nabla_{\bm{x}}\phi^{u}(\bm{x},\bm{y},t)\rho(\bm{y})\|_{\infty}\differential\bm{y}\nonumber\\
=& \smallint_{\mathbb{R}^2}\esssup_{(\bm{x},t)\in\mathbb{R}^{2}\times [0,\infty)]}\|\nabla_{\bm{x}}\phi^{u}(\bm{x},\bm{y},t)\|_{\infty}\rho(\bm{y})\differential\bm{y}\label{ineq:esssup}
\end{align}
where we used the Leibniz rule, triangle inequality, and that $\rho\geq 0$. Per assumption, $\nabla_{\bm{x}}\phi^{u}$ is uniformly bounded, and we have: \eqref{ineq:esssup} $\leq M\smallint_{\mathbb{R}^2}\rho(\bm{y})\differential\bm{y}=M$ for some constant $M>0$.

Condition (iv) holds because 
\begin{align*}
&\|\bm{f}^u\left(\bm{x},t,u(\bm{x},t),\rho\right) - \bm{f}^u\left(\bm{x},t,u(\bm{x},t),\widetilde{\rho}\right)\|_2\\
=& \|\nabla_{\bm{x}}\smallint 
_{\mathbb{R}^2}\phi^{u}(\bm{x},\bm{y},t)(\rho(\bm{y})-\widetilde{\rho}(\bm{y}))\differential\bm{y}\|_2\\
=& \|\smallint 
_{\mathbb{R}^2}\left(\nabla_{\bm{x}}\phi^{u}(\bm{x},\bm{y},t)\right)(\rho(\bm{y})-\widetilde{\rho}(\bm{y}))\differential\bm{y}\|_2\\
\leq& c\: {\rm{dist}}_{\rm{BL}}(\rho,\widetilde{\rho})\leq \kappa\: {\rm{dist}}_{\rm{TV}}(\rho,\widetilde{\rho})\quad\forall(\bm{x},t)\in\mathbb{R}^{2}\times[0,\infty)
\end{align*} 
for some constant $c>0$, $\kappa := 2c$, and the second to last inequality follows from $\nabla_{\bm{x}}\phi^{u}$ being bounded and Lipschitz.

Thus, we can guarantee the existence-uniqueness of sample path $\bm{x}_i(t)$, $i\in\llbracket n\rrbracket$, solving the interacting diffusion \eqref{ItoSDE}.

\subsection{Derivation of the Controlled Mean Field Model}
Our next result (Theorem \ref{Thm:MeanFieldPDEIVP}) derives the macroscopic mean field dynamics as a \emph{nonlinear} Fokker-Planck-Kolmogorov partial differential equation (PDE), and establishes the consistency of the mean field dynamics in the continuum limit vis-\`{a}-vis the finite population dynamics.
\begin{theorem}\label{Thm:MeanFieldPDEIVP}
Supposing \textbf{A1}, consider a population of $n$ interacting chiplets, where the $i$th chiplet position $\bm{x}_{i}\in\mathbb{R}^{2}$, $i\in\llbracket n\rrbracket$, evolves via \eqref{ItoSDE}. Denote the Dirac measure concentrated at $\bm{x}_{i}$ as $\delta_{\bm{x}_{i}}$ and let the random empirical measure $\rho^{n}:=\frac{1}{n} \sum_{i=1}^n \delta_{\bm{x}_{i}}$. Consider the empirical version of the dynamics \eqref{ItoSDE} given by
$$\differential\bm{x}_{i} = \bm{f}^u\left(\bm{x}_{i},t,u,\rho^{n}\right)\:\differential t + \sqrt{2\beta^{-1}}\:\differential\bm{w}_{i}(t),$$
with respective initial conditions $\bm{x}_{0i}\in\mathbb{R}^{2}$, $i\in\llbracket n\rrbracket$, which are independently sampled from a given PDF $\rho_{0}$ supported on a subset of $\mathbb{R}^{2}$.
  Then, as $n\rightarrow \infty$, almost surely $\rho^{n} \rightharpoonup \rho$ where the deterministic function $\rho$ is a PDF that evolves as per the macroscopic dynamics 
\begin{align}
\frac{\partial \rho}{\partial t}&=-\nabla \cdot(\rho \bm{f}^u) + \beta^{-1}\Delta\rho\nonumber\\
&= \nabla\cdot\left(\rho\nabla\left(\rho * \phi^{u}+\beta^{-1}(1+\log \rho)\right)\right),
\label{PDFdyn}
\end{align}
with the initial condition
\begin{align}
\rho(\cdot,t=0)=\rho_{0}\in\mathcal{P}\left(\mathbb{R}^2\right)\;(\text{given}).
\label{IC}    
\end{align}
\end{theorem}
\begin{proof}
To describe the dynamics of $\rho^n$ as $n\rightarrow \infty$, we start with investigating the time evolution of the quantity 
\begin{align}
    \left\langle\varphi, \rho^n\right\rangle := \frac{1}{n} \sum_{i=1}^n \varphi\left(\bm{x}_i\right)
\end{align}
for any compactly supported test function $\varphi \in C_b^2(\mathbb{R}^2)$.

Using Ito's rule, we have
\begin{align}
    \differential \varphi\left(\bm{x}_i\right)=L_{\rho^n} \varphi\left(\bm{x}_i\right) \differential t+\nabla \varphi^{\top}\left(\bm{x}_i\right) \sqrt{2\beta^{-1}} \differential \bm{w}_{i}
\end{align}
wherein the infinitesimal generator
\begin{align}
    L_\rho \varphi(\bm{x}):=\left\langle \bm{f}^{u}(\bm{x},t,u, \rho), \nabla_{\bm{x}} \varphi(x)\right\rangle+\beta^{-1} \Delta\varphi.
    \label{InfinitesimalGeneratorL}
\end{align}
Thus,
\begin{align}
   \differential\left\langle\varphi, \rho^n\right\rangle&=\frac{1}{n} \sum_{i=1}^n \differential \varphi\left(\bm{x}_i\right)\nonumber\\
& =\left\langle L_{\rho^n} \varphi, \rho^n\right\rangle \differential t+\frac{1}{n} \sum_{i=1}^n \sqrt{2 \beta^{-1}}\nabla  \varphi^{\top}\left(\bm{x}_i\right) \differential \bm{w}_{i}\nonumber \\
& :=\left\langle L_{\rho^n} \varphi, \rho^n\right\rangle \mathrm{d} t+\mathrm{d} M_t^n\label{dphirho}
\end{align}
where $M_t^n$ is a local martingale. 

Because $\varphi \in C_b^2(\mathbb{R}^2)$, we have $\left|\sqrt{2 \beta^{-1}}\nabla  \varphi^{\top}\left(\bm{x}_i\right)\right|\leq C$ 
uniformly for some $C>0$. Notice that the quadratic variation of the noise term in \eqref{dphirho} is
\begin{align*}
    \left[M_t^n\right]:=\frac{1}{n^2} \sum_{i=1}^n \int_0^t\left|\sqrt{2 \beta^{-1}}\nabla  \varphi^{\top}\left(\bm{x}_i(s)\right)\right|^2 \mathrm{~d} s \leq \frac{t C^2}{n},
\end{align*}
and using Doob's martingale inequality \cite[Ch. 14.11]{williams1991probability},
\begin{align}
    \!\!\!\!\!\mathbb{E}\left(\sup _{t \leq T} M_t^n\right)^2 \leq \mathbb{E}\left(\sup _{t \leq T}\left(M_t^n\right)^2\right) &\leq 4 \mathbb{E}\left(\left(M_t^n\right)^2\right) \nonumber\\
    &\leq 4 \mathbb{E}\left(\left[M_t^n\right]\right) \leq \frac{4 t C^2}{n}.\nonumber
\end{align}
Hence in the limit $n\rightarrow \infty$, the noise term in \eqref{dphirho} vanishes, resulting in a deterministic evolution equation.

For any $t>0$, we take $\{\rho^n\}_{n=1}^{\infty}$ to be the (random) elements of $\Omega=C([0, \infty), \mathcal{P}(\mathbb{R}^2))$, the set of continuous functions from $[0, \infty)$ into $\mathcal{P}(\mathbb{R}^2)$ endowed with the topology of weak convergence. Following the argument of Oelschl\"{a}ger \cite[Proposition 3.1]{oelschlager1985law}, the sequence $\mathbb{P}_n$ of joint PDFs on $\Omega$ induced by the processes $\{\rho^n\}_{n=1}^{\infty}$ , is relatively compact in $\mathcal{P}\left(\Omega\right)$, which is the space of probability measures on $\Omega$. Oelschl\"{a}ger's proof makes use of the Prohorov's theorem \cite[Ch. 5]{billingsley2013convergence}. The relative compactness implies that the sequence $\mathbb{P}_n$ weakly converges (along a subsequence) to some $\mathbb{P}$, where $\mathbb{P}$ is the joint PDFs induced by the limiting process $\rho$. By Skorohod representation theorem \cite[Theorem 6.7]{billingsley2013convergence}, the sequence $\{\rho^n\}_{n=1}^{\infty}$ converges $\mathbb{P}$-almost surely to $\rho$. Since the martingale term in \eqref{dphirho} vanishes as $n \rightarrow \infty$, we obtain
\begin{align}
\differential \left\langle\varphi, \rho\right\rangle=\left\langle L_\rho \varphi, \rho\right\rangle \differential t=\left\langle\varphi, L_\rho^* \rho\right\rangle \differential t  
\label{adjoint}
\end{align}
where  $L^*$ is the adjoint (see e.g., \cite[Ch. 2.3, 2.5]{pavliotis2014stochastic}, \cite[p. 278]{bogachev2015fokker}) of the generator $L$ given by \eqref{InfinitesimalGeneratorL}, and is defined as
\begin{align*}
L_m^* \rho(x,t):&=-\nabla \cdot(\rho \bm{f}^{u}(\bm{x},t,u, m))+\beta^{-1} \Delta\rho\\
&= \nabla\cdot\left(\rho\nabla\left(m * \phi^{u}+\beta^{-1}(1+\log \rho)\right)\right)
\end{align*}
where $m\in\mathcal{P}\left(\mathbb{R}^2\right)$. For any test function $\varphi \in C_b^2(\mathbb{R}^2)$, \eqref{adjoint} is valid almost everywhere, and therefore, $\rho$ is almost surely a weak solution to the nonlinear Fokker-Planck-Kolmogorov PDE initial value problem \eqref{PDFdyn}-\eqref{IC}.
\end{proof}
\noindent Notice that the Cauchy problem \eqref{PDFdyn}-\eqref{IC} involves a \emph{nonlinear nonlocal PDE} which in turn depends on control policy $u$.

The solution $\rho(\bm{x},t)$, $\bm{x}\in\mathbb{R}^2$, $t\in[0,\infty)$, for the Cauchy problem \eqref{PDFdyn}-\eqref{IC} is understood in weak sense. In other words, for all compactly supported smooth test functions $\theta\in C_{c}^{\infty}\left(\mathbb{R}^{2},[0,\infty)\right)$, the solution $\rho(\bm{x},t)$ satisfies
\begin{align}
\int_0^{\infty}\!\!\!\int_{\mathbb{R}^2}\!\!\left(\!\frac{\partial \theta}{\partial t} \!+\! L_{\rho}\theta\!\right)\!\rho\:\mathrm{d} \boldsymbol{x}\:\mathrm{d}t
\!+ \!\int_{\mathbb{R}^2}\!\rho_0(\boldsymbol{x}) \theta(\boldsymbol{x}, 0)\:\mathrm{d}\boldsymbol{x}=0
\label{WeakSoln}
\end{align}
where $L_{\rho}$ is defined as in \eqref{InfinitesimalGeneratorL}. The reason why $\rho$ satisfying \eqref{WeakSoln} for all $\theta\in C_{c}^{\infty}\left(\mathbb{R}^{2},[0,\infty)\right)$ is called a ``weak solution" of \eqref{PDFdyn}-\eqref{IC} is because such $\rho$ may not be sufficiently smooth to satisfy \eqref{PDFdyn}. In the next Section, we provide a variational interpretation of the solution for problem \eqref{PDFdyn}-\eqref{IC}.


\section{Chiplet Population Dynamics as Wasserstein Gradient Flow}\label{sec:WassGradFlow}
The structure of the PDE in \eqref{PDFdyn} motivates defining an \emph{energy functional}
\begin{align}
\Phi(\rho) &:= \Phicc(\rho) + \Phice(\rho)+\mathbb{E}_{\rho}\left[\beta^{-1}\log \rho\right]\nonumber\\ &=\mathbb{E}_{\rho}\left[\rho * \phi^u+\beta^{-1}\log \rho\right]
\label{equ:Phi}
\end{align}
where $\mathbb{E}_{\rho}$ denotes the expectation w.r.t. the PDF $\rho$, and
\begin{subequations}
\begin{align}
\Phicc(\rho) &:= \int_{\mathbb{R}^2 \times \mathbb{R}^2} \phicc^u(\bm{x},\bm{y})\rho(\bm{x}) \rho(\bm{y}) \differential\bm{x}\:\differential\bm{y},
\label{Phicc}\\
\Phice(\rho) &:= \int_{\mathbb{R}^2 \times \mathbb{R}^2} \phice^u(\bm{x},\bm{y})\rho(\bm{x}) \rho(\bm{y}) \differential\bm{x}\:\differential\bm{y}. \label{Phice}
\end{align}
\label{PhiccPhice}
\end{subequations}
In \eqref{equ:Phi}, the term $\mathbb{E}_{\rho}\left[\rho * \phi^u\right]$ quantifies the \emph{interaction energy} while the term $\beta^{-1}\mathbb{E}_{\rho}\left[\log \rho\right]$ (scaled negative entropy) quantifies the \emph{internal energy}. We have the following result.

\begin{theorem}\label{thm:WassGradFlow}
Let $\Phi:\mathcal{P}_{2}\left(\mathbb{R}^2\right)\mapsto\mathbb{R}$ be the energy functional given in \eqref{equ:Phi}. Then, \\
\noindent (i) the chiplet population dynamics given by \eqref{Controlledvectorfield}, \eqref{InteractionPotential}, \eqref{PDFdyn} is Wasserstein gradient flow of the functional $\Phi$, i.e.,
\begin{align}
\frac{\partial \rho}{\partial t}=-\nabla^W \Phi(\rho).
\label{WassGradFlow}
\end{align}
(ii) $\Phi$ is a Lyapunov functional that is decreasing along the flow generated by \eqref{PDFdyn}, i.e., $\frac{\differential}{\differential t}\Phi \leq 0$. 
\end{theorem}

\begin{proof}
(i) We start by noticing that the functional derivative
\begin{equation}
    \frac{\delta \Phi}{\delta \rho}=\rho * \phi^u+\beta^{-1}(1+\log \rho).
    \label{equ:derivation of Phi}
\end{equation}
Next, we rewrite \eqref{PDFdyn} as
\begin{equation}
\frac{\partial \rho}{\partial t}=\nabla \cdot \left(\rho \nabla\frac{\delta \Phi}{\delta \rho}\right),
\label{PDFdynAsWassGrad}
\end{equation}
which by definition \eqref{equ:Wasserstein gradient}, yields \eqref{WassGradFlow}.

\noindent (ii) To show that $\Phi$ is decreasing along the flow generated by \eqref{PDFdyn}, we find
\begin{equation}
    \begin{aligned}
    \frac{\differential}{\differential t} \Phi&=\int \frac{\delta \Phi}{ \delta \rho}  \: \frac{\partial \rho}{\partial t} \differential \bm{x}\\
    &\overset{\eqref{PDFdynAsWassGrad}}=\int \frac{\delta  \Phi }{ \delta \rho}\: \nabla \cdot \left(\rho \nabla \frac{\delta \Phi}{\delta \rho}\right) \differential \bm{x}\\
    &= -\int \bigg\langle \nabla \frac{\delta \Phi}{ \delta \rho}  ,\rho \nabla \frac{\delta \Phi}{\delta \rho} \bigg\rangle \differential \bm{x}\\
    &= -\int \bigg\langle \nabla  \frac{\delta \Phi}{\delta \rho}  , \nabla \frac{\delta \Phi}{\delta \rho} \bigg\rangle \rho \differential \bm{x}\\
    &= - \mathbb{E}_{\rho} \left[ \bigg\| \nabla \frac{\delta \Phi}{\delta \rho} \bigg\|_{2}^{2} \right] \leq 0.
    \end{aligned}
    \label{equ:dPhidt}
\end{equation}
In order to get from the second line to the third line of \eqref{equ:dPhidt}, we used the duality\footnote{In words, the gradient and the negative divergence are adjoint maps.}between the gradient and divergence operators, namely the fact that for differentiable scalar field $s(\bm{x})$ and vector field $\bm{v}(\bm{x})$, we have
\begin{align}
\langle \nabla s,\bm{v}\rangle_{L_{2}} + \langle s,\nabla\cdot\bm{v}\rangle_{L_{2}}=0,   
\label{graddivduality} 
\end{align}
where $\langle \bm{p}, \bm{q} \rangle_{L_{2}} := \int \langle \bm{p}, \bm{q}\rangle \differential\bm{x}$. Specifically, in \eqref{equ:dPhidt}, $s\equiv \frac{\delta \Phi}{\delta \rho}$ and $\bm{v}\equiv \rho \nabla \frac{\delta \Phi}{\delta \rho}$. 
\end{proof}
\begin{remark}\label{rmk:WassGradFlow}
Theorem \ref{thm:WassGradFlow} shows that for an admissible control policy $u\in\mathcal{U}$, the chiplet population dynamics \eqref{PDFdyn}-\eqref{IC} can be seen as gradient descent of the energy functional $\Phi$ on the manifold $\mathcal{P}_2\left(\mathbb{R}^2\right)$ w.r.t. the Wasserstein metric. We point out that the statement of Theorem \ref{thm:WassGradFlow} remains valid in the deterministic limit, i.e., when the noise strength $\sqrt{2\beta^{-1}}\downarrow 0$. In that case, the functional $\Phi$ in \eqref{equ:Phi} comprises of only the interaction energy term $\mathbb{E}_{\rho}\left[\rho * \phi^u\right]$, and $\frac{\delta\Phi}{\delta\rho} = \rho * \phi^u$. Other than this, the proof of Theorem \ref{thm:WassGradFlow} remains unchanged.
\end{remark}
\begin{remark}
In the recent systems-control literature, the Wasserstein gradient flow interpretations and related proximal algorithms \cite{caluya2019gradient,caluya2021wasserstein} for several linear and nonlinear Fokker-Planck-Kolmogorov PDEs in prediction and density control have appeared. New gradient flow interpretations have also appeared \cite{halder2017gradient,halder2018gradient,halder2019proximal} for well-known filtering equations. We next point out that the Wasserstein gradient flow interpretation deduced in Theorem \ref{thm:WassGradFlow} allows approximating the weak solution of \eqref{WassGradFlow} by recursive evaluation of a Wasserstein proximal operator on the manifold $\mathcal{P}_2\left(\mathbb{R}^{2}\right)$.
\end{remark}

\begin{theorem}\label{thm:WassProxRecursion}
For a given control policy $u\in\mathcal{U}$ and potentials \eqref{InteractionPotential}, let $\widehat{\Phi}(\varrho,\varrho_{k-1}):=\mathbb{E}_{\varrho}\left[\varrho_{k-1} * \phi^u+\beta^{-1}\log \varrho\right]$, $\varrho,\varrho_{k-1}\in\mathcal{P}_{2}(\mathbb{R}^2)$, $k\in\mathbb{N}$. Consider the Wasserstein proximal recursion:
\begin{align}
\!\!\varrho_{k} &= {\mathrm{prox}}^{W}_{\tau\widehat{\Phi}}\left(\varrho_{k-1}\right)\nonumber\\
&:= \underset{\varrho\in \mathcal{P}_2\left(\mathbb{R}^{2}\right)}{\arg\inf}\:\bigg\{\!\!\frac12\:W^{2}\left(\varrho,\varrho_{k-1}\right) + \tau\:\widehat{\Phi}(\varrho,\varrho_{k-1})\!\!\bigg\} 
\label{WassProx}
\end{align}
over discrete time $t_{k-1}:=(k-1)\tau$ with fixed step-size $\tau>0$, and with initial condition $\varrho_{0}\equiv\rho_0\in\mathcal{P}_2\left(\mathbb{R}^2\right)$. Let $\rho(\bm{x},t)$ be the weak solution of \eqref{WassGradFlow} for the same $u\in\mathcal{U}$ and the functional $\Phi$ given by \eqref{equ:Phi}-\eqref{PhiccPhice}. Using the sequence of functions $\{\varrho_{k-1}\}_{k\in\mathbb{N}}$ generated by the recursion \eqref{WassProx}, define an interpolation $\varrho_{\tau}:\mathbb{R}^{2}\times[0,\infty)\mapsto [0,\infty)$ as
$$\varrho_{\tau}(\bm{x},t) := \varrho_{k-1}(\bm{x},\tau)\quad \forall\:t\in [(k-1)\tau,k\tau), \quad k\in\mathbb{N}.$$ 
Then $\varrho_{\tau}(\bm{x},t)\xrightarrow[]{\tau\downarrow 0}\rho(\bm{x},t)$ in $L^{1}(\mathbb{R}^{2})$ for all $t\in[0,\infty)$.
\end{theorem}
\begin{proof}
Follows the development in \cite[Sec. 12.3--12.5]{laborde201712}.
\end{proof}
\begin{remark}\label{Remark:Numerical}
For a given control policy $u\in\mathcal{U}$, the Wasserstein proximal recursion \eqref{WassProx} can in turn be leveraged for numerically updating the PDFs over discrete time with a small step-size $\tau$. To illustrate Theorem \eqref{thm:WassGradFlow}, we fixed a linear control policy $u=\langle\bm{k},\bm{x}\rangle$ with gain $\bm{k}:=(8.5\times 10^{-3}, -1\times10^{-2})^{\top}$, and solved \eqref{WassProx} with $\tau=0.1$ via \cite[Algorithm 1]{caluya2019gradient} for $n=400$ uniformly spaced grid samples in the domain $\left[-4 \:\mathrm{mm}, 4 \:\mathrm{mm}\right]^2$ starting from an initial bivariate Gaussian $\rho_0=\mathcal{N}\left((0.5, 0.5)^{\top},0.1\bm{I}_{2}\right)$. Fig. \ref{fig:Phi} shows the corresponding decay of the energy functional $\Phi$ in \eqref{equ:Phi}-\eqref{PhiccPhice}, computed using these PDFs obtained from the Wasserstein proximal updates. As in \cite[Sec. III]{matei2021micro}, our simulation used capacitances $\Ccc(\|\bm{x}-\bm{y}\|_2),\Cce(\|\bm{x}-\bm{y}\|_2)$ in \eqref{Ucc}-\eqref{Uce} of the form $\sum\limits_{i=1}^{n} a_i\left[\erf((\|\bm{x}-\bm{y}\|_2 +\delta)/c_i)-\erf((\|\bm{x}-\bm{y}\|_2 -\delta)/c_i)\right]$ where $\erf(\cdot)$ denotes the error function, the parameters $a_i,c_i$ are sampled uniformly random in $[0,1]$, and $\delta$ (half of the electrode pitch) $=10$ micrometer. 
\end{remark}

\begin{figure}[t]
\centerline{\includegraphics[width=0.7\linewidth]{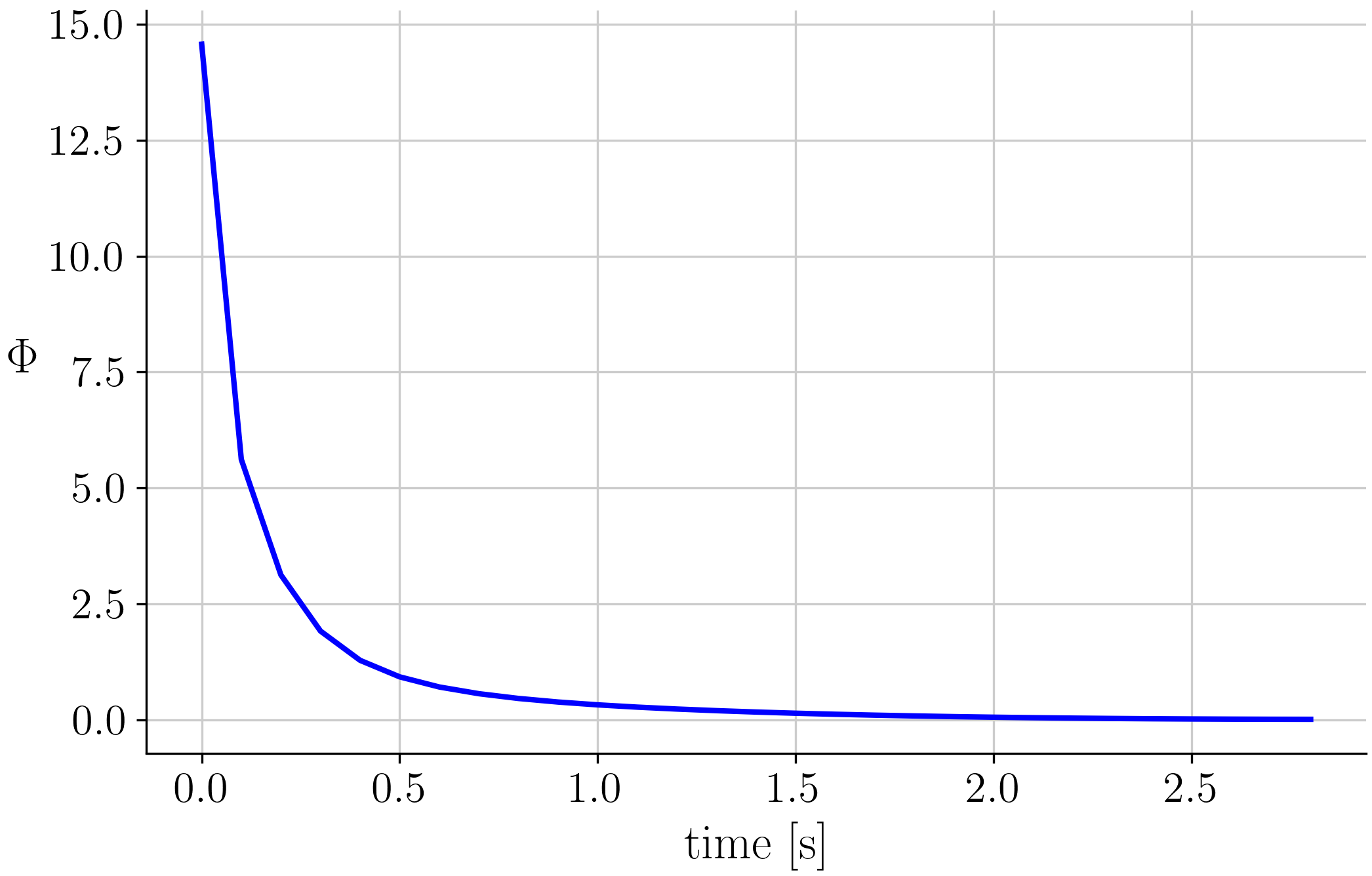}}
\vspace{.05in}
\caption{The energy functional $\Phi$ given by \eqref{equ:Phi}-\eqref{PhiccPhice} versus time for the simulation set up summarized in Remark \ref{Remark:Numerical}.}
\vspace*{-0.1in}
\label{fig:Phi}
\end{figure}

\section{Conclusions}\label{sec:conclusions}
We presented a controlled mean filed model for the population dynamics of chiplets, which are tiny (micron sized or smaller) particles immersed in a dielectric liquid, and are amenable to reshape into desired concentrations for micro-assembly applications. In such applications, an array of electrodes generate a space-time varying electric potential landscape, thereby strategically inducing the collective motion of the chiplet ensemble. Our derived model quantifies how exactly the two types of nonlocal nonlinear interactions (viz. chiplet-to-chiplet and chiplet-to-electrode) jointly induce a macroscopic dynamics in terms of the joint PDF evolution of the chiplet ensemble. Our results establish consistency of the model in a limiting sense, and demonstrate that the resulting PDF evolution can be seen as an infinite dimensional gradient descent of a Lyapunov-like energy functional w.r.t. the Wasserstein metric. 

While we focused our development for the derivation of the controlled mean field model, our future work will investigate the synthesis of optimal control of the chiplet joint PDF w.r.t. suitable performance objective that allows steering an initial joint PDF to a desired terminal joint PDF. Such feedback steering problems are generalized variants of the so-called Schr\"{o}dinger bridge problem \cite{chen2016relation}. We note that the feedback synthesis for density steering subject to a controlled mean field nonlocal PDE is relatively less explored but has started appearing in recent works; see e.g., \cite{chen2021density,sinigaglia2022optimal,zheng2022backstepping}. 

\bibliographystyle{IEEEtran}
\bibliography{references.bib}

\end{document}